\documentclass[12pt]{amsart}
\usepackage[colorlinks=true, pdfstartview=FitV, linkcolor=blue, citecolor=blue, urlcolor=blue]{hyperref}
\usepackage{amssymb}
\calclayout 

\def\quot#1#2{#1/\!\!/#2}
\def\gr{\operatorname{gr}}
\def\C{\mathbb {C}}
\def\R{\mathbb {R}}
\def\NN{\mathcal N}
\def\OO{\mathcal O}

\def\Z{\mathbb Z}

\def\M{\mathcal M}
\def\SL{\operatorname{SL}}

\def\GL{\operatorname{GL}}
\def\U{\operatorname{U}}
\def\SU{\operatorname{SU}}

\def\Sp{\operatorname{Sp}}
\def\SO{\operatorname{SO}}
\def\Orth{\operatorname{O}}

\def\Spin{\operatorname{Spin}}

\def\inv{^{-1}}
\def\lie#1{{\mathfrak #1}}
\def\lieg{\lie g}

\def\liek{\lie k}
\def\liem{\lie m}
\def\phi{{\varphi}}

\def\GG{\mathsf{G}}

\def\pr{{\operatorname{pr}}}
\def\codim{{\operatorname{codim}}}
\def\rk{\operatorname{rk}}

\def\J{\mathcal J}
\def\diff{\mathcal D}

\newcommand {\Spec} {\operatorname{Spec}}

\numberwithin{equation}{subsection}

\newtheorem{theorem}[subsection]{Theorem}
\newtheorem{lemma}[subsection]{Lemma}
\newtheorem{proposition}[subsection]{Proposition}
\newtheorem{corollary}[subsection]{Corollary}

\theoremstyle{definition}

\theoremstyle{remark}

\newtheorem{remark}[subsection]{Remark}

\newtheorem{example}[subsection]{Example}

\title{The Koszul complex of a moment map}

\author{Hans-Christian Herbig}
\address{Institut for Matematiske Fag\\
Aarhus Universitet\\
DK-8000 Aarhus C\\
Denmark}
\email{herbig@imf.au.dk}  

\author{Gerald W. Schwarz}
\address{Department of Mathematics\\
Brandeis University\\
Waltham, MA 02454-9110\\
USA}
\email{schwarz@brandeis.edu}
\keywords{moment map, Koszul complex}
\subjclass[2010] {Primary 20G20, 53D20, 57S15; Secondary 13D02}
\begin{document}
\begin{abstract}
Let $K\to\U(V)$ be a unitary representation of the compact Lie group $K$. Then there is a canonical moment mapping $\rho\colon V\to\liek^*$. We have the Koszul complex ${\mathcal K}(\rho,\mathcal C^\infty(V))$ of the component functions $\rho_1,\dots,\rho_k$  of $\rho$. Let $G=K_\C$, the complexification of $K$. We show that the Koszul complex is a resolution of the smooth functions on $\rho\inv(0)$ if and only if $G\to\GL(V)$ is $1$-large, a concept introduced in \cite{JAlg94,LiftingDO}. Now let $M$ be a symplectic manifold with a Hamiltonian action of  $K$. Let $\rho$ be a moment mapping and consider the Koszul complex given by the component functions of $\rho$. We show that the Koszul complex is a resolution of the smooth functions on $Z=\rho\inv(0)$ if and only if the complexification of each symplectic slice representation at a point of $Z$ is $1$-large.  
\end{abstract}

\maketitle

\section{Introduction}

This work is motivated by the Batalin-Fradkin-Vilkovisky (BFV) approach towards symplectic reduction of constrained systems and their quantization (cf.\ \cite{Stasheff92} and references therein). In its simplest incarnation, the BFV-method deals with symplectic reduction of moment maps at the zero level. The BFV-method is based on the Koszul complex of the moment map. If this complex is a resolution of the algebra of smooth functions on the zero fiber of the moment map, then the BFV-construction can be used to describe the Poisson algebra of smooth functions on the symplectic quotient. The BFV-algebra provides a natural starting point for attempts to quantize the symplectic quotient. For example, along these lines it has been proven that the symplectic quotient with respect to a compact connected Lie group admits a continuous formal deformation quantization, provided that the Koszul complex fulfills the aforementioned conditions \cite{BHP07}.
 
In \cite{HIP09} there is a  systematic examination of the Koszul complex of moment maps of unitary representations of tori. It turns out that exactness of the Koszul complex in homological degree $i\ge 1$  translates into a rank condition on the weight matrix. In order that the zeroth homology gives the algebra of smooth function on the zero fiber, the moment map has to satisfy a sign change condition.

Let $K$ be a compact Lie group with a Hamiltonian action on the symplectic manifold $M$ 
with moment mapping $\rho\colon M\to\liek^*$. For a closed subset $Z$  of $M$, we define $\mathcal C^\infty(Z)$, the
\emph{smooth functions on $Z$\/}, to be the restrictions  to $Z$ of the
smooth functions on $M$. We want to determine when the Koszul complex
$\mathcal K(\rho,\mathcal C^\infty(M))$ on the components  of $\rho$ is a
resolution of $\mathcal C^\infty( \rho\inv(0))$. This question is local, and as in \cite{AGJ90}, we use the symplectic slice theorem to reduce to the case that $M$ is a unitary representation $V$ of $K$ with the standard moment mapping $\rho$.     
We show that  $\mathcal K(\rho,\mathcal C^\infty(V))$ is a resolution of $\mathcal C^\infty(\rho\inv(0))$ if and only if the action of the complexification $G=K_\C$ on $V$ is $1$-large (see \S \ref{sec:reductive}). This condition was important in the second author's work on lifting of differential operators from the quotient $\quot VG$ to $V$. 

Here is an outline of the paper. In \S \ref{sec:reductive} we consider properties of the moment mapping $\mu$ associated to a representation $G\to\GL(V)$. This is  just the mapping $\mu\colon V\oplus V^*\to\lieg^*$ such that $\mu(v,v^*)(A)= v^*(A(v))$ for $(v,v^*)\in V\oplus V^*$ and $A\in\lieg$. We study the subscheme $Y$ of $V\oplus V^*$ which is defined by setting all the component functions of $\mu$ equal to zero. We recall results of Avramov\cite{Avramov}, Panyushev \cite{Panyushev} and the second author \cite{JAlg94, LiftingDO} which relate properties of $\OO(Y)$ to the geometry of the $G$-action on $V$. In \S \ref{sec:1large} we show that there are lots of $G$-modules $V$ which are $1$-large and give lists of some infinite families. In \S \ref{sec:unitarycase}
we consider unitary actions $K\to\U(V)$ of compact Lie groups $K$ and we show that the relevant Koszul complex $\mathcal K(\rho,\mathcal C^\infty(V))$ is a resolution of $\mathcal C^\infty(\rho\inv(0))$ if and only if $(V,K_\C)$ is $1$-large. In \S \ref{sec:example} we consider an example of \cite{AGJ90} in our language, showing what can go wrong when $1$-largeness fails. In \S \ref{sec:sympslice} we consider a Hamiltonian action $K\times M\to M$ of a compact Lie group $K$ with   moment mapping $\rho$ and recall the symplectic slice theorem.  In \S \ref{sec:manifolds} we combine the slice theorem and the results of the previous sections to give our necessary and sufficient condition for $\mathcal K(\rho,\mathcal C^\infty(M))$ to be  a resolution of $\mathcal C^\infty(\rho\inv(0))$.

\vspace{2mm}
\noindent{\emph{Acknowledgements.}} The authors would like to thank Srikanth Iyengar, Lucho Avramov and Markus Pflaum for stimulating discussions and moral support. We also acknowledge financial support  and hospitality provided  by the Center for the Quantization of Moduli Spaces (QGM) at Aarhus University, the University of Nebraska at Lincoln and the University of Cologne.

\section{Linear actions of reductive groups} \label{sec:reductive}

Let $G$ be a complex reductive group and let $V$ be a $G$-module. Then the algebra of invariant polynomial functions $\OO(V)^G$ is finitely generated. Set $\quot VG:=\Spec\OO(V)^G$. We have a canonical map $\pi\colon V\to \quot VG$ dual to the inclusion $\OO(V)^G\subset \OO(V)$. Each fiber of $\pi$ contains a unique closed orbit. Let $Gv$ be a closed orbit. Then the isotropy group $H:=G_v$ is reductive. The isotropy groups at points of $Gv$ all lie in the conjugacy class $(H)$ of $H$. Let $Z$ denote $\quot VG$ and let $Z_{(H)}$ denote the image of the closed orbits with isotropy group in $(H)$. The $Z_{(H)}$ are a finite stratification of $Z$ where each $Z_{(H)}$ is locally closed, smooth and irreducible. There is a unique open stratum which we denote as $Z_\pr$.   We call any element of the associated conjugacy class a \emph{principal isotropy group\/} and associated closed orbits are called \emph{principal orbits}. We set $V_\pr:=\pi\inv(Z_\pr)$. 
We say that $V$ is \emph{stable\/} if $V_\pr$ consists of closed orbits. We say that \emph{$V$ has FPIG\/} (respectively  \emph{$V$ has TPIG\/}) if its principal isotropy groups are finite (respectively trivial). In either case $V$ is stable. We say that $V$ is \emph{$m$-principal\/} if $V$ has FPIG and $\codim V\setminus V_\pr\geq m$, $m\geq 1$. If $V$ has FPIG, then it is automatically $1$-principal.

Let $V_{(j)}$ denote the points of $V$ where $\dim G_v=j$. Then the $V_{(j)}$ are locally closed subsets of $V$.
We say that the $G$-action on $V$ is \emph{locally free\/} if $V_{(0)}\neq\emptyset$.
 Let $c_j$ denote $\dim V_{(j)}-\dim G+j$. If $V_{(j)}$ is irreducible, then $c_j$ is the transcendence degree of  $\C(V_{(j)})^G$.   We say that $V$ is \emph{$m$-modular\/} if $V_{(0)}\neq \emptyset$ and $c_j+m\leq c_0$ for all $j\geq 1$. Note that this differs slightly from the definition given in \cite{LiftingDO} where we assume that $V$ has FPIG. If $V$ has FPIG, then $c_0=\dim\quot VG$. We say that $V$ is \emph{$m$-large\/} if it is $m$-principal and $m$-modular.  If $V$ is $0$-modular or has FPIG, then it is locally free. Conversely, if $V$ is locally free and $G$ is semisimple, then $V$ has FPIG \cite{Popov}.

There is a natural moment mapping   $\mu\colon V\oplus V^*\to\lieg^*$ where for $A\in\lieg$ and $(v,v^*)\in V\oplus V^*$ we have $\langle\mu(v,v^*),A\rangle=v^*(A(v))$. Choosing a basis $A_1\dots,A_k$ of $\lieg$ we may think of $\mu$ as a $k$-tuple  of functions $\mu_1,\dots,\mu_k \in\OO(V\oplus V^*)$.  Let $\J$ denote the ideal of the $\mu_i$ in $\OO(V\oplus V^*)$.

\begin{lemma}\label{lem:0modular}
\begin{enumerate}
\item Suppose that the $G$-action is locally free. Then $\mu$ has rank $\dim G$ on an open dense subset of $V\oplus V^*$.
\item The $G$-action on $V$ is $0$-modular if and only if the $\mu_i$ are a regular sequence in $\OO(V\oplus V^*)$.
\end{enumerate}
\end{lemma}

\begin{proof}
Part (1) is trivial, since for fixed $v\in V_{(0)}$ and $A\in\lieg$, $\mu(v,\cdot)(A)$ is the linear function $v^* \mapsto v^*(A(v))$ where $\lieg(v)$ has dimension $k=\dim G$. Part (2) is \cite[Proposition 9.4]{LiftingDO}.  
\end{proof}

It is useful to think of $\OO(V\oplus V^*)$ as $\gr\diff(V)$ where $\diff(V)$ is the algebra of algebraic differential operators on $V$ and $\gr$ is taken with respect to the filtration by order of differentiation. Let $Y\subset V\oplus V^*$ be the subscheme corresponding to $\J$ .  We consider $\lieg$ as differential operators of order $1$ on $V$.

\begin{theorem}\label{thm:2large}
Set $Q:=\diff(V)/(\diff(V)\lieg)$. Then
\begin{enumerate}
\item $(\diff(V)\lieg)^G$ is an ideal in $\diff(V)^G$.  Hence $Q^G$ is naturally an algebra of differential operators on $\OO(\quot VG)$.
\item If $V$ is $0$-modular, then $\gr Q\simeq\OO(Y)$ and $Y$ is a complete intersection.
\item If $V$  is $1$-large, then  %
$Y$ is reduced and irreducible, i.e., $\J$ is prime.
\item If $V$ is $2$-large, then %
$Y$ is normal.
\item If $V$ is $2$-large, then $Q^G\simeq\diff(\quot VG)$.
\end{enumerate}
\end{theorem}

 \begin{proof}
One can see that for any $A\in\diff(V)^G$ and $B\in(\diff(V)\lieg)^G$, $AB$ is obviously in $(\diff(V)\lieg)^G$ and $BA$ is also since $A$ commutes with $\lieg$. Hence we have (1). For (2) we can appeal to \cite[Proposition 8.11]{LiftingDO} and Lemma \ref{lem:0modular}. For (3)  let $Y_j$ denote $(V_{(j)}\times V^*)\cap Y$ for $j\geq 0$. Then each $Y_j$ is a vector bundle over $V_{(j)}$ with fiber dimension $\dim V-\dim G+j$ and $Y_0$ is smooth.
Let $0\neq f\in\OO(V)^G$ such that $f$ vanishes on $V\setminus V_\pr$. Then $f$ is not a zero divisor of $\OO(Y)$ \cite[Lemma 9.7]{LiftingDO}, hence $\OO(Y)\to\OO(Y)_f=\OO(Y_f)$ is injective. Since $Y_f$ is open in $Y_0$, it is smooth, hence reduced. Thus  $Y$ is reduced. Now  every irreducible component of $Y$ has codimension  $k=\dim G$.  For $j\geq 1$ we have that
$$
\dim Y_j=\dim V_{(j)}+\dim V-\dim G+j=c_j+\dim V\leq 2\dim V-\dim G-1=\dim Y-1
$$
since $V$ is $1$-modular. Thus $Y_j$ has codimension greater than $k$, hence its closure is not dense in any irreducible component of $Y$. Since $V_{(0)}$ is irreducible, so is $Y_0$. Thus $Y=\overline{Y_0}$ is irreducible  and we have (3). For (4), we know that  $Y_0$ is smooth and since $V$ is $2$-large, there are invariants $f_1$ and $f_2$ which are a regular sequence for $\OO(Y)$ which vanish off $V_\pr\times V^*$ \cite[Lemma 9.7]{LiftingDO}, hence they vanish off of $Y_0$. Thus  the singularities of $Y$  are in codimension two and we have normality of $Y$  giving (4). Part (5) is \cite[Theorem 0.4(1)]{LiftingDO}.
\end{proof}	

\begin{remark}\label{rem:Ynotirred} Suppose that $V$ has FPIG.
It is clear from the proof above that if $V$  fails to be $1$-modular, then $Y$ is not irreducible.  Also,   there is a unique irreducible component $\overline{Y_0}$ of $Y$ whose projection to $V$ is dense.
\end{remark}

\begin{remark}
One can also combine work of D.\ Panyushev on the Jacobian module \cite{Panyushev} and work of L.\ Avramov on symmetric algebras \cite{Avramov} to obtain  %
part of (2) above and (3) and (4) with ``large'' replaced by ''modular.'' 
We need to recall the notion of the  Jacobian module. Consider $\lieg$ as a subset of $V^*\otimes V$ and let $R$ denote $\OO(V)$. Then we have the following map of free $R$-modules:
$$
 R\otimes \lieg\xrightarrow{\widehat{\mu}} R\otimes V,\quad f\otimes\sum_i \xi_i\otimes v_i \mapsto \sum_i f\xi_i\otimes v_i
 $$
where the $v_i$ are in $V$,  the $\xi_i$ are in $V^*\subset R$ and $\sum_i \xi_i\otimes v_i\in\lieg$. The \emph{Jacobian module} $E$ is defined to be the cokernel of $\widehat{\mu}$. Note that $\widehat{\mu}$, in coordinates, is a matrix of linear forms.
It is known (cf.\ \cite[Theorem 1.3]{Panyushev}) that the coordinate ring $\C[Y]$ of the zero fibre $Y=\mu^{-1}(0)\subset V\oplus V^*$ is isomorphic to the symmetric algebra $S_R(E)$. Let $I_t(\widehat{\mu})$ be the ideal in $R$ that is generated by the $t\times t$-minors of the matrix $\widehat{\mu}$. For $d\geq 0$  consider the condition 
\begin{eqnarray*}
(\mathcal F_d): \qquad \operatorname{grade}I_t(\widehat{\mu})\ge \rk(\widehat{\mu})-t+1+d,\quad \text{ for }1\le t\le \rk(\widehat{\mu}).
\end{eqnarray*}
 In  case $V$ is locally free, we can apply \cite[Theorem 2.4]{Panyushev}) to conclude that condition $(\mathcal F_d)$ is equivalent to the representation being $d$-modular. Applying \cite[Proposition 3.(2)]{Avramov}, it follows that $S_R(E)$ is a complete intersection if $(\mathcal F_0)$ holds. By \cite[Theorem 1.8]{Panyushev}, $\widehat{\mu}$ is injective, and hence, $\operatorname{pd}_R(E)=1$. 
Then \cite[Proposition 5]{Avramov} shows that if $V$ is $1$-modular, then  $S_R(E)\cong \C[Y]$ is a domain. Conversely, if $S_R(E)\cong \C[Y]$ is a domain and $V$ is locally free, then it is  $1$-modular. 
Finally, by  \cite[Proposition 6]{Avramov}, if $V$ is $2$-modular  then $S_R(E)\cong \C[Y]$ is factorial, hence normal.
\end{remark}

 \section{Modules which are $1$-large}\label{sec:1large}
 
 We consider modules which are $1$-large for various classes of groups. It turns out that $1$-largeness usually  holds. From \cite[Proposition 11.5]{LiftingDO} we have
 
 \begin{proposition}
Let $(V,G)$ be $1$-large. Then so is $V\oplus V'$ for any $G$-module $V'$ .
 \end{proposition}
 
 \begin{theorem}\label{thm:torus}
 Let $T\simeq(\C^*)^n$ be a torus and let $V$ be a faithful $T$-module. Then $V$ is stable if and only if it is $1$-large.
 \end{theorem}
 
 \begin{proof} We may assume that   $T$ acts stably. Then the principal isotropy groups are trivial and $V$ is $1$-principal.
 By a theorem of Vinberg \cite{Vinberg} it follows that  $V$ is   $1$-modular, hence $1$-large.
  \end{proof}
 
   \begin{remark}  
Let $T\to\GL(V)$ be a torus action, not necessarily effective. Then $V$ is  stable if and only if $0$ is in the  relative interior of the convex hull of the weights  (\cite[Lemma 2]{WehlauPop}, \cite[Lemma 1]{WehlauCons}). If $S\to U(V)$ is a unitary action of a compact torus $S$, the condition on the convex hull of the weights is one of the equivalent conditions  in   \cite[Proposition 2.2]{HIP09} for the action of $S$ to satisfy the ``generating hypothesis.'' Thus these conditions are equivalent to the stability of the action of the complexification $T=S_\C$ on $V$.
    \end{remark}
    
For non-toral groups things are a bit more difficult, but there are some  general results, mainly coming from our investigations of  differential operators on quotients \cite{JAlg94,LiftingDO}.

Let $R_j$ denote the representation of $\SL_2$ on the binary forms of degree $j$, $j\in\Z^+$. 
\begin{theorem}\label{thm:sl2}(cf.\ \cite[Example 7.14]{AGJ90})
Let $V$ be a nonzero $\SL_2$-module with $V^{\SL_2}=(0)$. Then   $V$ is $1$-large, except when $V\simeq R_1$, $V\simeq 2R_1$ or $V\simeq R_2$.
\end{theorem}

\begin{proof}
In \cite[Theorem 11.9]{LiftingDO} we listed the $V$ which are not $2$-large. The techniques we used   show  that all the entries on the list are $1$-large, except for $R_1$, $2R_1$ and $R_2$.
\end{proof}

For the classical representations of the classical groups one can also be precise. Again the techniques can be found in \cite[\S 11]{LiftingDO}. Let $(\C^7,\GG_2)$ and $(\C^8,\Spin_7)$ denote the irreducible representations of dimensions $7$  and $8$, respectively.

\begin{theorem}
 The following representations are $1$-large.
 \begin{enumerate}
\item $(V,G)=(p\C^n+q(\C^n)^*,\GL_n)$, $p$, $q\geq n$.
\item $(V,G)=(p\C^n+q(\C^n)^*,\SL_n)$, $p+q\geq 2n-1$.
\item $(V,G)=(p\C^n,\SO_n)$ or $(p\C^n,\Orth_n)$, $p\geq n-1$.
\item $(V,G)=(p\C^{2n},\Sp_{2n})$, $p\geq 2n+1$.
\item $(V,G)=(p\C^7,\GG_2)$, $p\geq 4$.
\item $(V,G)=(p\C^8,\Spin_7)$, $p\geq 5$.
\end{enumerate}
 In each case, the inequalities on $p$ and $q$ are the minimum possible.
\end{theorem}

From \cite[Corollary 11.6]{LiftingDO} we have 
\begin{theorem}
Let $G$ be connected semisimple and $V$ a $G$-module such that $V^G=(0)$ and such that $\lieg$ acts faithfully on any irreducible component of $V$. Then, up to isomorphism, there are only finitely many $V$ which are not $1$-large.
\end{theorem}
We say that a $G$-module is \emph{coregular\/} if $\OO(V)^G$ is a polynomial ring, equivalently; $\quot VG$ is smooth. This is a rather rare occurrence. For $G$ simple the list of irreducible coregular representations was determined in \cite{KPV}.
From \cite{JAlg94} we have
\begin{theorem}
Let  $G$ be a  connected simple algebraic group and let  $V$ be an irreducible  noncoregular $G$-module. Then $V$ is $2$-large.
\end{theorem}

\section{The unitary case}\label{sec:unitarycase}

Let $V$ be a unitary representation of the compact Lie group $K$.   We have a moment mapping $\rho\colon V\to\liek^*$, as follows. For $v\in V$ and $A\in \liek$ we set  %
$\rho(v)(A)=(-i/2)\langle A(v),v\rangle$ where $\langle\, ,\rangle$ is the hermitian form on $V$ (which by our convention is conjugate linear in the second argument). Since $iA$ is hermitian symmetric, $\rho$ is real valued.  

 Let $G$ denote the complexification of $K$. Then $G$ also acts on $V$ (the image of $G$ in $\GL(V)$ is the Zariski closure of the image of $K$). We have that $\lieg=\lie k+i\lie k$.  Consider the complexification $V_\C:=V\otimes_\R\C$. Then $V_\C\simeq V\oplus V^*$ as a representation of $G$.   One sees this as follows. Let $J$ denote the complex structure on $V$. Then $V_\C$ decomposes into the $\pm i$-eigenspaces  $V_{\pm}$ of $J$. We have $V_+=\{v-iJv\mid v\in V\}$ and $V_-=\{v+iJv\mid v\in V\}$. Thus $V_\C=V_+\oplus V_-$ where $V_+\simeq V$ by the obvious mapping. We have a nondegenerate pairing of $V_+$ and $V_-$ which sends $(v-iJv,v'+iJv')$ to $\langle v,v'\rangle$. %
Thus $V_-\simeq V^*$. We have the moment mapping $\mu$ on $V\oplus V^*$ where $\mu(v,v^*)(A)=v^*(A(v))$ for $A\in\lieg$, $v\in V$ and $v^*\in V^*$. Thus, in $V_\C$, $\mu$ has the formula $\mu(v-iJv,v'+iJv')(A)=\langle A(v),v'\rangle$. Now we have an embedding of $V$ into $V_\C$ which sends $v$ to $v\otimes 1=\frac 12(v-iJv,v+iJv)\in V_+\oplus V_-$. Thus, up to a  scalar,  $\mu(v\otimes 1)$, considered as an element of $i\liek^*$,  is just $\rho(v)$. If $\rho(v)=0$, then $\mu(v\otimes 1)$ vanishes on $i\liek$, hence $\mu(v\otimes 1)=0$. Thus the zero set $Y$ of $\mu$ intersected with $V\simeq V\otimes 1$ is the zero set   of $\rho$. This zero set, denoted $\M$, is the \emph{Kempf-Ness set of $V$}. It is $K$-stable and a real algebraic cone in $V$.
 It is a basic fact that an orbit $Gv$ is closed if and only if it intersects $\M$, and then the intersection is a $K$-orbit. Hence $\quot VG\simeq\M/K$. Moreover, for $v\in \M$, $G_v=(K_v)_\C$.  If the quotient of $V$ by $G$ is not a point, then $\M$ contains nonzero points.

\begin{proposition}\label{prop:goodrealpoint}
Suppose that $(V,G)$ has FPIG. Then the points of $\M$ where $\rho$ has rank $k$ form an open dense subset of $\M$ (classical topology).
\end{proposition}

\begin{proof}
It is clear that the set of points of $\M$ where the rank of $\rho$ is $k=\dim G$ is open and contains the principal points $\M_\pr:=V_\pr\cap \M$.  Let $Z$ denote $\quot VG$. Then $Z_\pr$ is the image of $\M_\pr$ and $Z_\pr$ is open and dense in $Z\simeq\M/K$. It follows that  $\M_\pr$ is  dense in $\M$. 
\end{proof}

We consider $\M$ as a subset of  $Y$ as above.

\begin{corollary}
Suppose that $(V,G)$ is $1$-large. Then $\M$ is Zariski dense in $Y$.
\end{corollary}
 	
\begin{proof}
By Theorem \ref{thm:2large} we know that $Y$ is reduced and irreducible.  Let $v\in\M$ such that $\rho$ has rank $k$ at $v$. Then $v$ is a smooth point of $\M$ which shows that $\M$ has codimension $k$ in $V$. Thus the complexification of $\M$ has complex codimension $k$, which shows that $\M$ is Zariski dense in $Y$.  
\end{proof}

Let $X$ be a real analytic subset of $\R^n$ and let $x\in X$. Let $f_1,\dots,f_m$ be germs of analytic functions generating the ideal of $X$ at $x$. We say that \emph{$X$ is coherent at $x$\/} if the $f_i$ generate the ideal of $X$ at $x'$ for $x'$ in a sufficiently small neighborhood of $x$ in $\R^n$. We say that $X$ is \emph{coherent} if it is coherent at every $x\in X$.

\begin{corollary}
Suppose that $(V,G)$ is $1$-large. Then $\M$ is coherent and its real analytic ideal is generated by the component functions of $\rho$.
\end{corollary}

\begin{proof}
We have our generators $\mu_1,\dots,\mu_k$ of the ideal $\mathcal J$ of $Y$ in $V\oplus V^*$. They restrict to elements of the ideal of $\M$ in $V$. If $f$ is a real analytic germ at $x\in \M$ which vanishes on $\M$, then its complexification  $f_\C$ vanishes on $Y$ since the germ of $\M$ at $x$ contains the germ of an open subset of $\M$ consisting of smooth points.  Thus $f_{\C}$ lies in the ideal of the $\mu_i$ which generate the ideal of $Y$ at all points. Hence at points near $x$ in $\M$, $f$ lies in the ideal of the $\rho_i$.   Thus  $\M$ is  coherent and  its real analytic ideal is generated by the restrictions $\rho_i$ of the $\mu_i$.
\end{proof}

\begin{corollary}\label{cor:idealgenerators}
Suppose that $(V,G)$ is $1$-large. Then $\rho_1,\dots,\rho_k$ generate the ideal of $\M$ in $\mathcal C^\infty(V)$.
\end{corollary}
\begin{proof}
 We just saw that the real analytic ideal of $\M$ is generated by the $\rho_i$.  By \cite[Ch.\ VI, 3.10]{Malgrange} or \cite[Ch.\ VI, 4.2]{Tougeron}, coherence of $\M$ is equivalent to the ideal of $\M$ in $\mathcal C^\infty(V)$ being generated  by the $\rho_i$.   
\end{proof}

\begin{proposition} \label{prop:koszul}  Suppose that $(V,G)$ is $1$-large. Then
the Koszul complex ${\mathcal K}(\rho,\mathcal C^\infty(V))$ of $\rho_1,\dots,\rho_k$ is acyclic, i.e., has homology concentrated in degree zero.
\end{proposition}

\begin{proof}  
It is enough to consider germs at a point $x\in\M$. We know that the germs of the $\rho_i$ form a regular sequence, i.e., the relations  are just the Koszul relations. But the inclusion  of the germs of regular functions to the germs of analytic functions is faithfully flat as is the inclusion of the germs of analytic functions to the germs of smooth functions  \cite[Ch.\ VI, 1.12]{Malgrange} or \cite[Ch.\ VI, 1.3]{Tougeron}. Thus the smooth relations of the germs of the $\rho_i$ are locally generated by the Koszul relations. Hence this is true globally.
\end{proof}

\begin{corollary}\label{cor:koszul} Suppose that $(V,G)$ is $1$-large. Then
the Koszul complex ${\mathcal K}(\rho,\mathcal C^\infty(V))$ is a resolution of the algebra of smooth functions on $\M$.
\end{corollary}

 We have the converse of the corollary.

\begin{proposition}\label{prop:converse}
If the Koszul complex ${\mathcal K}(\rho,\mathcal C^\infty(V))$ of $\rho_1,\dots,\rho_k$ is a resolution of the smooth functions on $\M$, then $V$ is $1$-large.
\end{proposition}
\begin{proof}
The $\rho_i$ have to be a regular sequence, hence so are the corresponding polynomial functions $\mu_i$ on $V\oplus V^*$, $i=1,\dots,k$.
Thus $V$ is $0$-modular and  all irreducible components of $Y$ have dimension $2\dim V-\dim G$. 
Since the  ideal of  $\M$ is generated by the $\rho_i$, $\M$ is Zariski dense in $Y$. 

Suppose that $V$ is not stable. Then $\quot VG$ has dimension less than $\dim V-\dim G$, and  $\M$ must have real dimension at most $\dim G+2\dim\quot VG<2\dim V-\dim G $. But  $Y$ has dimension $2\dim V-\dim G$, hence $\M$ cannot be dense in $Y$. Thus $V$ is stable and it follows that  $V$ has FPIG and is $1$-principal.

Let $\M_0$ denote the complement in $\M$ of the principal orbits $\M_\pr$. Then $\M_0$ is a proper algebraic subset of $\M$, hence it is not dense in any irreducible component of $Y$. Let $Y_0$ be as in  Remark \ref{rem:Ynotirred}. Then   $\M_\pr\subset Y_0$ where $\overline{Y_0}$ is an irreducible component of $Y$. Hence   $Y$ is irreducible and by Remark \ref{rem:Ynotirred} this  implies that $V$ is $1$-modular.
\end{proof}

\section{An example}\label{sec:example}
 
We rework in our language an  example of \cite{AGJ90} where $(V,G)$ is not $1$-large and 
$Y$ is not reduced or  irreducible.

\begin{example}\label{ex:sl2} (Cf.\ \cite[Example 7.13]{AGJ90})
Let us consider  unitary $K:=\SU_2$-modules $V$. It follows from Theorem \ref{thm:sl2} that   most $V$ are $1$-large. Let us look at an example where this fails. Let $V=2\C^2$ with the usual unitary action of $K$. Then $(V,G)=(2\C^2,\SL_2(\C))$ where $(V,G)$ has TPIG but is only $0$-modular.  We have the moment mapping $\mu\colon V\oplus V^*\to\mathfrak g^*$ with zero set $Y$.  Let $\pr_1\colon Y\to V$ denote projection onto the first factor. The complement of $V_\pr$ is   the null cone $\NN$, defined to be $\pi\inv(\pi(0))$. It consists of pairs of vectors $(v_1,v_2)\in V$ which are linearly dependent. This is a subset of $V$ of codimension one and for $v\in \NN':=\NN\setminus\{0\}$, the dimension of $G_v$ is 1. Thus $\pr_1\inv(\NN')\to\NN'$ is a vector bundle   with fiber dimension 2, and the closure is an irreducible component $Y'$ of $Y$. The points of $Y$ lying over $0\in V$ are in $Y'$. By Remark \ref{rem:Ynotirred}, we also have one more irreducible component $\overline{Y_0}$, the unique irreducible component whose projection to $V$ is dense. Now consider the mapping $V\to V_+\oplus V_-$, $v\mapsto v\otimes 1=v_++v_-$. If $v_+\in\NN(V_+)$, then $v\in\NN(V)$ since $V\to V_+$ is a $G$-isomorphism. But $\NN(V)\cap\M=\{0\}$. Thus $Y'\cap\M=\{0\}$ is not dense in $Y'$ while $\M$ is dense in $\overline{Y_0}$.
 \end{example}

 \begin{example}\label{ex:Y0notreduced}
 We consider cases where $Y$ is not reduced.   Let $K\subset L$ be  compact Lie groups where $\dim K<\dim L$. Suppose that we have a unitary $L$-module $V$ such that $\OO(V)^G=\OO(V)^H$ where $G=K_\C$ and $H=L_\C$. Let $\mu_G$ (respectively $\mu_H$) be the moment mapping for $G$ (respectively $H$). Let $Y_G:=\mu_G\inv(0)$ and $Y_H:=\mu_H\inv(0)$. Then $Y_G$ is not reduced. One sees this as follows. Since $K\subset L$, the Kempf-Ness set $\M_L$ for the $H$-action is a subset of the Kempf-Ness set $\M_K$  of the $G$-action. But $\M_K/K\simeq\M_L/L\simeq \quot VG\simeq\quot VH$. Thus we must have that $\M_K=\M_L$. This says that  $\mu_H=0$ and $\mu_G=0$   give the same zero set when restricted to $V$. Hence the irreducible components of $Y_G$ in the Zariski closure of $\M_K$ are not reduced and so $Y_G$ is not reduced.
   
In  Example \ref{ex:sl2} above,  let $\overline{\SL}_2(\C)$ denote another copy of $\SL_2(\C)$ acting on a copy $\overline{\C}^2$ of $\C^2$. Consider the action of $\SO_4(\C)=(\SL_2(\C)\times\overline{\SL}_2(\C))/\pm I$ on $\C^2\otimes \overline{\C}^2$. Then the $\SO_4(\C)$-invariants are the same as the $\SL_2$-invariants, hence $\overline{Y_0}$ is not reduced.
  \end{example}
  
  \section{The symplectic slice theorem} \label{sec:sympslice} 
  
  We need to recall the symplectic slice theorem. It will allow us to apply our results above to Hamiltonian actions of compact groups on manifolds.
  
  Let $(M,\omega)$ be a Hamiltonian $K$-manifold, $K$ a compact Lie group, with moment map $\rho:M\to \mathfrak k^*$. Let $Z:=\rho^{-1}(0)$ be the zero fiber and let $z\in Z$. The symplectic slice at $z$ is defined as $V:=T_z(Kz)^\omega/(T_z(Kz)\cap T_z(Kz)^\omega)$. Here $T_z(K.z)^\omega$ denotes the perpendicular space relative to $\omega$. The subspace $V$ can be identified with a subspace of the slice representation at $z$ and the symplectic form $\omega_z$ restricted to $V$ is non-degenerate and $K_z$-invariant.  We have the canonical moment map $\rho_V$ on $V$ with image in $ \mathfrak k_z^*$. Choose a  $K_z$-equivariant splitting $\mathfrak k=\mathfrak k_z\oplus \mathfrak m$ so that $\liek^*=\liem^*\oplus\liek_z^*$. We set 
\[Y:=K\times^{K_z}(\mathfrak \liem^*\times V).\]
There exists a canonical $K$-invariant symplectic form on $Y$ and moment mapping $\rho_Y\colon Y\to\liek^*$.
\begin{theorem} [\cite{GS84,Marle85}]\label{thm:sympslice}
Let $(M,\omega)$, $z$, etc.\ be as above.
\begin{enumerate}
\item The action of $K$ on $Y$ is Hamiltonian with moment map:
\begin{eqnarray*}
\rho_Y: Y\to \mathfrak k^*,\quad \rho_Y(k,\lambda,v)=\operatorname{Ad}^*_{k\inv}(\lambda+ \rho_V(v)),
\end{eqnarray*}
where $\operatorname{Ad}^*_{k\inv}$ denotes the coadjoint action of $k\inv\in K$ on $\mathfrak k^*$.
\item There exists a $K$-invariant open neighborhood $U_1$ of $Kz$ in $M$, a  $K$-invariant open neighborhood $U_2$ of $[e,0]$ in $Y$ and  a $K$-equivariant symplectomorphism  $\varphi:U_1 \to U_2$.
\item The pull-back of $\rho_Y$ by $\varphi$ is $\rho$.
\end{enumerate}
\end{theorem}

\begin{remark}\label{rem:hermitian}
Let $V$ be as above. Since $\omega_z$ is non-degenerate and $K_z$-invariant on $V$, there is a $K_z$-invariant hermitian form $\langle\,,\rangle$ on $V$ whose associated $2$-form (the imaginary part of $\langle\,,\rangle)$ is $\omega_z$. Note that the moment map  $\rho_V$ is defined by $\rho_V(v)(A)= (1/2)\omega_z(Av,v)$, $v\in V$, $A\in\liek_v$  which is the same %
as $(-i/2)\langle A(v),v\rangle$, so the canonical moment mapping on $V$ is the one that we have been using.

The hermitian structure on $V$ gives us a complex structure on $V$ and we  have an action of $G_z=(K_z)_\C$ on $V$.  
The isomorphism class of the representation of $G_z$ is uniquely determined by the $K_z$-action. Thus  it makes sense to ask if $(V,G_z)$ is $1$-large, etc. The answer is independent of the complex structure that we choose.
\end{remark}

 \section{Actions on manifolds}\label{sec:manifolds}
 
 \begin{theorem} Let $K$ be a compact Lie group acting in a Hamiltonian way on a symplectic manifold  $M$, and let $\rho:M\to \mathfrak k^*$ be a moment map for the action.
 Then the Koszul complex $\mathcal K(\rho,\mathcal C^\infty(M))$
 is a resolution of the ring $\mathcal C^\infty(Z)$ of smooth functions on $Z=\rho^{-1}(0)$ if and only if for every $z\in Z$ the complexification of the symplectic slice representation at $z$  is $1$-large.
\end{theorem}

\begin{proof}
Let $z\in Z$. By the formula in Theorem \ref{thm:sympslice} we only have to prove the result for $Y:=K\times^{K_z}(\liem^*\times V)$  with the moment map  $\rho_Y$ given there. 
Now the zero set of $\rho_Y$ is $K\times^{K_z}\rho_V\inv(0)$. Thus the theorem follows from Corollary \ref{cor:koszul} and Proposition \ref{prop:converse} (see Remark \ref{rem:hermitian}).
\end{proof}

\def\cprime{$'$}
\providecommand{\bysame}{\leavevmode\hbox to3em{\hrulefill}\thinspace}
\providecommand{\MR}{\relax\ifhmode\unskip\space\fi MR }
\providecommand{\MRhref}[2]{%
  \href{http://www.ams.org/mathscinet-getitem?mr=#1}{#2}
}
\providecommand{\href}[2]{#2}

\end{document}